\documentclass[a4paper]{amsart}
\usepackage{amsmath,amsthm,amssymb,amscd}
\usepackage[arrow,matrix]{xy}
\usepackage{enumerate}
\usepackage{mathrsfs}

\theoremstyle{plain}
\numberwithin{equation}{section}
\newtheorem{thm}{Theorem}[section]

\newtheorem{prop}[thm]{Proposition}
\newtheorem{cor}[thm]{Corollary}
\newtheorem{lem}[thm]{Lemma}
\theoremstyle{definition}
\newtheorem{dfn}[thm]{Definition}
\newtheorem{exm}[thm]{Example}

\newtheorem{rem}[thm]{Remark}

\def\dim{\mathop{\mathrm{dim}}\nolimits}
\def\Lie{\mathop{\mathrm{Lie}}\nolimits}

\def\Ad{\mathop{\mathrm{Ad}}\nolimits}

\def\torsor{E_{\sigma}\to\Gamma(\sigma)^{\mathrm{gp}} \backslash D_{\sigma}}

\def\GsDs{\Gamma(\sigma)^{\mathrm{gp}} \backslash D_{\sigma}}
\def\Gs{\Gamma(\sigma)^{\mathrm{gp}}}

\def\Es{E_{\sigma}}

\def\CC{\mathbb{C}}
\def\BB{\mathbf{B}}
\def\QQ{\mathbb{Q}}
\def\RR{\mathbb{R}}
\def\ZZ{\mathbb{Z}}
\def\PP{\mathbb{P}}
\def\Gr{\mathrm{Gr}}

\def\calB{\mathcal{B}}

\def\calM{\mathcal{M}}

\def\cl{\mathrm{cl}}

\def\frakg{\mathfrak{g}}

\def\frakh{\mathscr{H}}

\def\nn{\mathbf{n}}
\def\hh{\mathbf{h}}

\def\Re{\mathop{\mathrm{Re}}\nolimits}
\def\Im{\mathop{\mathrm{Im}}\nolimits}
\def\bs{\backslash}
\def\gp{\mathrm{gp}}

\def\Aut{\mathop{\mathrm{Aut}}\nolimits}

\def\Span{\mathrm{span}}

\def\even{\mathrm{ev}}
\def\ev{\mathrm{ev}}
\def\odd{\mathrm{od}}
\def\od{\mathrm{od}}
\def\tr{\mathrm{tor}}
\def\Sat{S}
\def\scrH{\mathscr{H}}
\def\l{\left}
\def\r{\right}
\def\rt-m{\sqrt{m}}
\def\00{\mathbf{0}}
\def\Ker{\mathop{\mathrm{Ker}}\nolimits}

\begin{document}
\title[Boundary of cycle spaces]{Boundaries of cycle spaces\\ and degenerating Hodge structures}
\author[T.~Hayama]{Tatsuki Hayama}
\address{Department of Mathematics, National Taiwan University, Taipei 106, Taiwan}
\curraddr{Mathematical Sciences Center, Tsinghua University, Haidian District, Beijing 100084, China}
\email{tatsuki@math.tsinghua.edu.cn}
\thanks{Supported by National Science Council of Taiwan}
\date{\today}
\subjclass[2000]{32G20, 14D07.}
\keywords{degenerating Hodge structure; partial compactification of period domain; cycle space}
\begin{abstract}
We study a property of cycle spaces in connection with degenerating Hodge structures of odd-weight, and construct maps from some partial compactifications of period domains to the Satake compatifications of Siegel spaces.
These maps are a generalization of the maps from the toroidal compactifications of Siegel spaces to the Satake compactifications.
We also show the continuity of these maps for the case for the Hodge structure of Calabi-Yau threefolds with $h^{2,1}=1$.    
\end{abstract}
\maketitle
\section{Introduction}
A cycle space is a space which parametrizes compact manifolds of a domain in a flag manifold.
This is studied in the field of complex geometry (e.g. Fels, Huckleberry and Wolf \cite{FHW}).
Our aim is to study period domains for pure Hodge structures by using the theory of cycle spaces and apply it to the study of degenerating Hodge structures.
To be more specific we aim to study the partial compactifications of period domains introduced by Kato and Usui \cite{KU}.
We will show results about the partial compactifications for some degenerations by using a property of cycle spaces.
Especially for the case for the Hodge structures of Calabi-Yau threefolds with $h^{2,1}=1$ we will show an explicit calculation and a further result for this case.

\subsection{Period domains and the partial compactifications}
By the classical works of Griffiths \cite{G}, a variation of Hodge structure gives a horizontal map to the period domain called the period map . 
Kato and Usui \cite{KU} constructed partial compactifications so that period maps can be extended for some degenerating Hodge structures.
They showed these partial compactifications are moduli spaces of some degenerating Hodge structures called {\it log Hodge structures}.
If we consider the Hodge structures of curves or K3 surface, period domains are Hermitian symmetric domains.
For Hermitian symmetric domains there are several ways to make compactifications.
Especially toroidal partial compactifications \cite{AMRT} are coincide with the partial compactifications of \cite{KU} if period domains are Hermitian symmetric.
These partial compactifications are given by fans, then the properties of fans have an effect on the properties of geometry on the partial compactifications.
\cite{AMRT} gave the constructions of fans which give compactifications. 

For a general period domain, \cite{KU} showed fundamental geometric properties of the partial compactifications of period domains using log geometry.
The partial compactifications are not analytic spaces but analytic spaces with slits called {\it log manifolds}. 
In contrast with Hermitian symmetric case we do not have a construction of fans to make compactifications, and it is expected not to exist such fans, if period domains are not Hermitian symmetric.
Moreover we do not have a construction of a {\it complete} fan \cite[Definition 12.6.1]{KU}.
Then the construction problem of fans is still remained.

In this paper we will discuss about the partial compactifications related to the special fans for period domains for odd-weight Hodge structures.
We will show that some properties on these special partial compactifications. 
\subsection{Cycle spaces}
We review cycle spaces briefly following \cite{FHW}.
Let $D$ be a period domain.
Then $D$ is an open orbit of the flag manifold $\check{D}$ called the compact dual.
Here the real Lie group $G$ act on $D$ transitively.
Fixing a base point $F_0$ of $D$, the isotropy subgroup $L_0$ is compact, which is maximally compact if, and only if, $D\cong G/ L_0$ is a Hermitian symmetric domain.
Taking the maximally compact subgroup $K_0$ containing $L_0$ we have the orbit $C_0=K_0F_0$, which is a compact submanifold contained in $D$.
The cycle space $\calM_{D}$ is a set of all $gC_0$ with $g\in G_{\CC}$ which is contained in $D$.

In this paper we treat $D$ for odd-weight Hodge structures.
By shifting, we may assume the weight is $-1$.
In this case we have
\begin{align}\label{groups}
G\cong Sp(n,\RR),\quad L_0\cong \prod_{j\geq 0} U(n_j),\quad K_0\cong U(n),\quad C_0\cong K_0/L_0
\end{align}
where $\sum_jn_j=n$ and $\{n_j\}_j$ depends on the Hodge numbers (see \cite{CMP}).
Here $G/K_0$ is isomorphic to the Siegel space $\scrH$ of degree $n$ and we have the real analytic map
$$D\cong G/L_0\to G/K_0\cong\scrH.$$
Through this map, a $\ZZ$-Hodge structure whose Hodge filtration is in $D$ is corresponding to an abelian variety, which coincide with the Weil intermediate Jacobian.
Now $K_0=L_0$ if, and only if, $h^{p,-p-1}=0$ for $p\neq 0,1$ (the case for Hodge structures of curves).
Moreover $\calM_D\cong \scrH$ if  $K_0=L_0$.

We are mainly interested in the case where $L_0\neq K_0$, i.e. $D$ is not a symmetric space.
In this case the cycle space can be written by $\calM_D\cong \scrH\times \bar{\scrH}$ where $\bar{\scrH}$ is its complex conjugate by \cite{FHW}.
Since $\scrH$ is a familiar object, $\calM_D$ is easier to study than $D$ itself. 
In our last paper \cite{H} we show some properties of $\calM_D$ in connection with $SL(2)$-orbits.
In this paper we will give a generalization of this result in Proposition \ref{cyc-bd}.

\subsection{Degenerating Hodge structures}
We may regard a nilpotent orbit as a degenerating Hodge structure by Schmid \cite{S} and Cattani, Kaplan and Schmid \cite{CKS}.
A nilpotent orbit is given by a rational nilpotent cone of $\Lie{G}$ and an orbit in $\check{D}$ defined by this nilpotent cone.
Boundary points of the partial compactifications of period domains are corresponding to nilpotent orbits.
   
Let $\sigma$ be a nilpotent cone, and let $\BB(\sigma)$ be the set of all $\sigma$-nilpotent orbits.
For a fan $\Sigma$ of nilpotent cones we denote by $D_{\Sigma}$ the union of $\BB(\sigma)$ for $\sigma\in \Sigma$, where the component $\BB(\{0\})$ for the $0$-cone is $D$.
Taking a subgroup $\Gamma$ of the discrete group $G_{\ZZ}$ which is compatible with $\Sigma$, $\Gamma\bs D_{\Sigma}$ is the partial compactification related to $\Sigma$. 
We will define an even-type (resp. odd-type) nilpotent cone $\sigma$ in Definition \ref{ev-od} and the Satake boundary component $\BB_S(\sigma)$ of $\scrH$ corresponding to $\sigma$.
By using the properties of cycle spaces (Proposition \ref{cyc-bd}) we will construct a map
$$p^{\ev}:\BB(\sigma)\to \BB_{S}(\sigma),\quad (\text{resp. }p^{\od}:\BB(\sigma)\to\overline{\BB_{S}(\sigma)}).$$
If $D$ is a Siegel space, $p^{\ev}$ coincides with the map $\zeta$ given by \cite{CCK} (see Example \ref{siegel-exm}).
\begin{thm}\label{main1}
Let $\Sigma$ be an even-type (resp. odd-type) fan (Definition \ref{ev-od}), and let $\Gamma$ be a subgroup of $G_{\ZZ}$ which is compatible with $\Sigma$.
Then we have the map $p^{\ev}$ (resp. $p^{\od}$) from the partial compactification $\Gamma\bs D_{\Sigma}$ defined by Kato and Usui \cite{KU} to the Satake compactification $\Gamma\bs \scrH_S$ (resp. $\overline{\Gamma\bs \scrH}_S$).
If $D$ is a Siegel space, $p^{\ev}$ coincides with the map from the toroidal partial compactification to the Satake compactification. 
\end{thm}
Moreover an even-type (resp. odd-type) cone $\sigma$ gives a nilpotent orbit for $\scrH$ (resp. $\bar{\scrH}$).
We denote by $\BB_{\tr}(\sigma)$ the boundary component of $\scrH_{\sigma}$.
We will define the map
$$\tilde{p}^{\ev}:\BB(\sigma)\to \BB_{\tr}(\sigma)\quad (\text{resp. }\tilde{p}^{\od}:\BB(\sigma)\to \overline{\BB_{\tr}(\sigma)}).$$
If $D$ is a Siegel space, $\tilde{p}^{\ev}$ is the identity.
Then we will obtain the following theorem:
\begin{thm}
Let $\Sigma$ be an even-type (resp. odd-type) fan, and let $\Gamma$ be a subgroup of $G_{\ZZ}$ which is compatible with $\Sigma$.
Then we have the map $\tilde{p}^{\ev}:\Gamma\bs D_{\Sigma}\to\Gamma\bs\scrH_{\Sigma}$ (resp.  $\tilde{p}^{\od}: \Gamma\bs D_{\Sigma}\to\overline{\Gamma\bs\scrH}_{\Sigma}$), which factors through the map of Theorem \ref{main1}.
If $D$ is a Siegel space, $\tilde{p}^{\ev}$ is the identity.
\end{thm}
Finally we have the following commutative diagrams:
\begin{align*}
\xymatrix{
\Gamma\bs D_{\Sigma_{\ev}}\ar@{->}[dr]_{p^{\ev}}\ar@{->}[r]^{\tilde{p}^{\ev}} & \Gamma\bs \scrH_{\Sigma_{\ev}}\ar@{->}[d]^{\zeta}\\
&\Gamma\bs\scrH_{\Sat}
}\quad
\xymatrix{
\Gamma\bs D_{\Sigma_{\od}}\ar@{->}[dr]_{p^{\od}}\ar@{->}[r]^{\tilde{p}^{\od}} & \overline{\Gamma\bs \scrH}_{\Sigma_{\od}}\ar@{->}[d]^{\bar{\zeta}}\\
&\overline{\Gamma\bs\scrH}_{\Sat}
}
\end{align*}
where $\Sigma_{\ev}$ is an even-type fan and $\Sigma_{\od}$ is an odd-type fan.

\subsection{The (1,1,1,1)-case}
We will study these maps in detail for the case of the Hodge structures of Calabi-Yau threefolds with $h^{2,1}=1$.
This case is geometrically corresponding to the quintic mirror family or the Borcea-Voisin mirror family (see \cite{GGK}), and the nilpotent orbits are explicitly described and classified by Kato and Usui \cite{KU}.
Then we know the constructions of the fans $\Sigma_{\ev}$ and $\Sigma_{\od}$ and these nilpotent orbits in this case.  
By using it we will describe the maps $p^{\ev},p^{\od},\tilde{p}^{\ev}$ and $\tilde{p}^{\od}$ and show the following proposition:
\begin{prop}
In this case, $\tilde{p}^{\ev}$ and $\tilde{p}^{\od}$ (therefore $p^{\ev}$ and $p^{\od}$) are continuous.
\end{prop}
Here the topology of $\Gamma\bs D_{\Sigma}$ is the strong topology, which makes these continuous.
Remark that the even-types and the odd-types are not parallel.
In fact $p^{\od}$ is not surjective although $p^{\ev}$ is and $\Sigma_{\od}$ is not a part of a fan of a toroidal compactification although $\Sigma_{\ev}$ is.
Especially the following diagram holds for the even-case (see Remark \ref{rem-ev}):
\begin{align*}
\xymatrix{
\Gamma\bs D_{\Sigma_{\ev}}\ar@{->}[dr]_{p^{\ev }}\ar@{->}[r]^{\tilde{p}^{\ev}} & \Gamma\bs \scrH_{\Sigma_{\ev}}\ar@{->}[d]\ar@{->}[r]^{\subset}& \Gamma\bs \scrH_{\Sigma_{\tr}}\ar@{->}[ld]^{\zeta}\\
&\Gamma\bs\scrH_{\Sat}
}
\end{align*}
where $\Gamma\bs \scrH_{\Sigma_{\tr}}$ is a toroidal compactification of the Siegel space $\scrH$ of degree $2$.

\subsection{Further problems}
How to construct $\Sigma_{\ev}$ and $\Sigma_{\od}$ and the property of $p^{\ev},p^{\od},\tilde{p}^{\ev}$ and $\tilde{p}^{\od}$ beyond the above case is unknown.
We expect that $p^{\ev}$ and $p^{\od}$ have a good geometric property similar to the Siegel case. 
Since the toroidal compactifications and the Satake compactifications are well known, we expect that these maps are helpful to study the geometry of $\Gamma\bs D_{\Sigma_{\ev}}$ and $\Gamma\bs D_{\Sigma_{\od}}$.

On the other hand Green, Griffiths and Kerr \cite{GGK2} recently study Mumford-Tate domains.
Kerr and Pearlstein \cite{KP} investigate a relationship between boundary components of Mumford-Tate domains and the Kato-Usui boundary components.
I expect that our study of Kato-Usui boundary components can be applied to the study of boundary components of Mumford-Tate domains.  
\section*{Acknowledgment}
This work was done during a visit of the author to Johns Hopkins University for the activity of JAMI in February 2012. The author is grateful for the hospitality and the support.
The author is thankful to professors Radu Laza, Gregory Pearlstein and Steven Zucker for their valuable advice and warm encouragement. 
 




\section{Even-type and odd-type Degenerations}
Let $(H_{\ZZ},F_0,\langle\; ,\; \rangle)$ be a polarized Hodge structure of weight $-1$ where $H_{\ZZ}$ is $\ZZ$-module, $F_0$ is a filtration of $H_{\CC}:=H_{\ZZ}\otimes \CC$ and $\langle\; ,\; \rangle$ is a non-degenerate alternating bilinear form on $H_{\ZZ}$.
We have the period domain $D$ and its compact dual $\check{D}$ by \cite{G}.
Here $D$ is written as a homogeneous space on which the real Lie group $G$ acts.
As in (\ref{groups}), $G\cong Sp(n,\RR)$ and the isotropy subgroup $L_0$ is isomorphic to $\prod_{j\geq 0} U(n_j)$ where $n_j$ is the Hodge number $h^{j,-j-1}$.
We write $\Lie{G}=\frakg$.
\subsection{$\RR$-split LMHS with $N^2=0$}
Let $N\in\frakg$ be a nilpotent with $N^2=0$ and let $F\in\check{D}$. 
We have the monodromy weight filtration $W(N)$ such that
$$W(N)_0=H_{\RR},\quad W(N)_{-1}=\Ker{N},\quad W(N)_{-2}=\Im{N}$$
(which is shifted by $-1$ from the original definition of the weight filtration).
We say $(W(N),F)$ is a limiting mixed Hodge structure (LMHS) if the following properties hold:
\begin{itemize}
\item $(W(N),F)$ is a mixed Hodge structure;
\item $N:\Gr^{W(N)}_0\stackrel{\sim}{\to}\Gr^{W(N)}_{-2}$ gives a $(-1,-1)$-morphism of Hodge structure;
\item $\langle\bullet ,N^j\bullet\rangle$ gives a polarization for $\Gr^{W(N)}_{j-1}$ ($j=0,1$).
\end{itemize}
Now we have the Deligne decomposition $H_{\CC}=\bigoplus_{p+q=0,-1,-2} I^{p,q}$ letting
$$I^{p,q}=F^{p}\cap W(N)_{p+q ,\CC}\cap(\bar{F}^q\cap W(N)_{p+q,\CC}+\sum_{j\geq 2}\bar{F}^{q-j+1}\cap W(N)_{p+q-j,\CC}),$$
which gives Hodge decomposition on each graded quotient part.
We assume that the LMHS is $\RR$-split, i.e. the Deligne decomposition is defined over $\RR$.
By \cite[Lemma 3.12]{CKS}, $\exp{(zN)}F\in D$ for $\Im{z}>0$.
\begin{lem}\label{exp-zN}
Let $v\in I^{p,-p}$.
For $z\in \CC$ with $\Im{z}>0$, $\exp{(zN)}v$ is in the $(p,-p-1)$-component of the Hodge decomposition with respect to $\exp{(zN)}F$.
\end{lem}
\begin{proof}
Now $\exp{(zN)}v\in \exp{(zN)}F^p$.
Moreover $\bar{v}\in F^{-p}$ and $N\bar{v}\in F^{-p-1}$ since the LMHS is $\RR$-split.
Here
\begin{align*}
&N\bar{v}=\exp{(zN)}(N\bar{v})\in \exp{(zN)}F^{-p-1},\\
&\bar{v}-zN\bar{v}=\exp{(zN)}\bar{v}-2zN\bar{v}\in \exp{(zN)}F^{-p-1}.
\end{align*}
Then
\begin{align*}
\exp{(zN)}v&=v+zNv\\
&=\overline{\bar{v}-zN\bar{v}}+2\Re{(z)}Nv\in \overline{\exp{(zN)}F^{-p-1}}.
\end{align*}
\end{proof}
Let $H_{\CC}=\bigoplus_p H^{p,-p-1}$ be the Hodge decomposition with respect to $\exp{(zN)}F$ for $\Im{z}>0$.
By the above lemma, 
\begin{align*}
&e^{zN}I^{p,-p}\subset H^{p,-p-1},\quad e^{zN}I^{-p,p}\subset H^{-p,p-1},\\
&e^{\bar{z}N}I^{-p,p}=\overline{e^{zN}I^{p,-p}}\subset H^{-p-1,p},\quad e^{\bar{z}N}I^{p,-p}=\overline{e^{zN}I^{-p,p}}\subset H^{p-1,-p}. 
\end{align*}
We write 
\begin{align}\label{decomp_123}
H^{p,-p-1}_1=I^{p,-p-1},\quad H^{p,-p-1}_2=e^{zN}I^{p,-p},\quad H^{p,-p-1}_3=e^{\bar{z}N}I^{p+1,-p-1}.
\end{align}
Then the $(p,-p-1)$-component has the decomposition
\begin{align}\label{p-decomp}
H^{p,-p-1}=H^{p,-p-1}_1\oplus H^{p,-p-1}_2\oplus H^{p,-p-1}_3.
\end{align}
\begin{rem}
By \cite[Lemma 6.24]{S}, $H_{\CC}$ can be decomposed into the direct sum of the subspaces which are invariant and irreducible with respect to the Hodge structure.
Here every irreducible subspace is isomorphic to $H(p)\otimes S(-2p-1)$, with $p\leq -1$, or $E(p,q)\otimes S(-p-q-1)$, with $p+q\leq -1$.
The relationship between this decomposition and the decomposition (\ref{p-decomp}) is written as follows: 
$$H^{p,-p-1}_1\oplus H^{-p-1,p}_1$$
is the direct sum of the $(E(p,-p-1)\otimes S(0))$-type components,
\begin{align*}H^{p,-p-1}_2\oplus H^{p-1,-p}_3\oplus H_2^{-p,p-1}\oplus H_3^{-p-1,p}\\
=I^{p,-p}\oplus I^{-p,p}\oplus I^{p-1,-p-1}\oplus I^{-p-1,p-1}
\end{align*}
is the direct sum of the $(E(p-1,-p-1)\oplus S(1))$-type components for $p\geq 1$, and 
$$H_2^{0,-1}\oplus H^{-1,0}_3=I^{0,0}\oplus I^{-1,-1}$$
is the direct sum of the $(H(-1)\oplus S(1))$-type components.
\end{rem}
\subsection{Cycle spaces and SL(2)-orbits}
Let $(N,F)$ be a pair which generates a LMHS with $N^2=0$. 
By \cite[Proposition 2.20]{CKS} there exists $\delta\in L^{-1,-1}_{\RR}(W(\sigma),F)$ uniquely such that $(W(\sigma),e^{-i\delta}F)$ is a $\RR$-split LMHS.
We write $\hat{F}=e^{-i\delta}F$.
By the SL(2)-orbit theorem (\cite[Theorem 5.13]{S}, \cite[\S 3]{CKS}), there exists a Lie group homomorphism $\rho: SL(2,\CC)\to G_{\CC}$ defined over $\RR$ and a holomorphic map $\phi:\PP^1\to \check{D}$ satisfying the following conditions:
\begin{itemize}\label{sl2}
\item $\rho(g)\phi(z)=\phi(gz)$;
\item $\phi(0)=\hat{F}$;
\item $\rho_*(\nn_-)=N$;
\item $Hv=(p+q+1)v$ for $v\in I^{p,q}$ where $\rho_*(\hh)=H$;
\item $\rho_*:\mathfrak{sl}(2,\CC)\to\frakg_{\CC}$ is a $(0,0)$-morphism of Hodge structure where $\frakg$ (resp. $\mathfrak{sl}(2,\RR)$) has a Hodge structure of weight $0$ relative to $\phi(i)$ (resp. $i$),
\end{itemize}
where $\{\nn_-,\hh,\nn_+\}$ are the standard generators of $sl(2,\CC)$.
We fix $F_0=\phi(i)=\exp{(iN)}\hat{F}$ as a base point of $D$.
We write
$$X=\frac{1}{2}(iN-H+iN^+)$$
where $N^+=\rho_* (\nn_+)$.
Then $X$ is in the $(-1,1)$-component of the Hodge decomposition of $\frakg_{\CC}$ with respect to $F_0$ and $X^2=0$.
Let $H_{\CC}=\bigoplus H^{p,-p-1}$ be the Hodge decomposition with respect to $F_0$.
By Lemma \ref{exp-zN}, for $v\in I^{p,-p}$ we have
$$u=\exp{(iN)}v\in H^{p,-p-1}$$
Then $Xu\in H^{p-1,-p}$.
We denote by $\|\bullet \|$ the norm induced by the polarization with respect to $F_0$.
Scaling $v$, we may assume $\| u\|=1$.
\begin{lem}\label{X-act}
$Xu=-\exp{(-iN)}v$ and $\|Xu\|=1$.
\end{lem}
\begin{proof}
By the property of the $sl_2$-triple,
\begin{align*}
&N^+Nv=v,\quad N^+N\bar{v}=\bar{v},\quad N^+v=N^+\bar{v}=0,\\
&Hv=v,\quad H\bar{v}=\bar{v},\quad HNv=-Nv,\quad HN\bar{v}=-N\bar{v}.
\end{align*}
Then 
\begin{align*}
Xu&=\frac{1}{2}(iN-H+iN^+)(v+iNv)\\
&=-v+iNv=-\exp{(-iN)}v.
\end{align*}

Next we show $\|Xu\|=1$.
Let $a=\langle v,\bar{v}\rangle$, $b=\langle Nv,\bar{v}\rangle$, $c=\langle v,N\bar{v}\rangle$ and $d=\langle Nv,N\bar{v}\rangle$.  
Then by the orthogonality and the positivity
\begin{align*}
&\langle u,\bar{u}\rangle =a+ib-ic+d=i^{-2p-1},\quad \langle u,\overline{Xu}\rangle =-a-ib-ic+d=0,\\
&\langle Xu, \bar{u}\rangle =-a+ib+ic+d=0.
\end{align*}
Since $v\in \hat{F}^p$ and $\bar{v}\in \hat{F}^{-p}$, $a=0$.
Therefore the simultaneous equation induces $d=0$, $b-c=i^{-2p-2}$ and
$\langle Xu,\overline{Xu}\rangle =a-ib+ic+d=-i^{-2p-1}$.
Then $\|Xu\|=i^{2p-1}\langle Xu,\overline{Xu}\rangle=1$.
\end{proof}
Then $X$ gives the isomorphism
\begin{align*}
X:H^{p,-p-1}_2\to H^{p-1,-p}_3;\quad \exp{(iN)}v\mapsto -\exp{(-iN)}v.
\end{align*}
Therefore we have
\begin{align}\label{val-1}
i^{2p+1}\langle\exp{(zX)}u,\overline{\exp{(zX)}u}\rangle =1-|z|^2.
\end{align}
If $v'\in I^{p,-p}$ is orthogonal to $v$ for $\langle\bullet ,N\bar{\bullet}\rangle$,
\begin{align}\label{val-2}
\langle\exp{(zX)}u,\overline{\exp{(zX)}u'}\rangle =0
\end{align}
where $u'=\exp{(iN)}v'$.


For the Hodge numbers $\{h^{p,-p-1}\}_p$, we define
\begin{align*}
f^p_{\ev}=\sum_{\substack{r\geq p,\\ r\text{: even}}}h^{r,-r-1},\quad f^p_{\od}= \sum_{\substack{r\geq p,\\ r\text{: odd}}}h^{r,-r-1}.
\end{align*}
The maximally compact subgroup $K_0$ is isomorphic to the unitary group $U(n)$ as in (\ref{groups}), and the orbit $C_0=K_0F_0$ is a compact submanifold of $D$ by \cite[Lemma 5.1.3]{FHW}.
The cycle space is defined by
$$\calM_D=\{gC_0\; |\; gC_0\subset D,\; g\in G_{\CC} \}.$$
By \cite[Lemma 5.1.3]{FHW}, $\calM_D$ is an open subset of the complex manifold
$$\calM_{\check{D}}=\{gC_0\; |\; g\in G_{\CC} \}.$$
If $K_0=L_0$, i.e. $D$ is a Siegel space, $C_0$ is the base point $F_0$ and $\calM_D=D$.
We describe $\calM_D$ for the case where $K_0\neq L_0$ following \cite[\S 5.5B]{FHW}. 
Here $C_0$ can be written as
$$C_0=\{ F\in D \; |\;\dim{(F^p\cap H^{\even})}=f^p_{\even},\; \dim{(F^p\cap H^{\odd})}=f^p_{\odd}\}$$
where
$$H^{\ev}=\bigoplus_{p:\text{ even}} H^{p,q},\quad H^{\od}=\bigoplus_{p: \text{ odd}}H^{p,q}.$$
Let $V$ and $W$ be complementary $\langle \; ,\;\rangle$-isotropic $n$-dimensional subspaces of $H_{\CC}$, and let
$$C_{V,W}=\{ F\in \check{D} |\;\dim{(F^p\cap V)}=f^p_{\even},\dim{(F^p\cap W)}=f^p_{\odd}\}.$$
Here $C_0=C_{H^{\even},H^{\odd}}$ and $gC_{V,W}=C_{gV,gW}$ for $g\in G_{\CC}$.
By using this, the cycle space is described as
\begin{align*}
\calM_{D}=\{C_{V,W}|\; V\gg 0\text{ and }W\ll 0\text{ for }-i\langle \bullet ,\bar{\bullet} \rangle \}.
\end{align*}
Now $G/K_0$ is isomorphic to the Siegel space $\scrH$ of degree $n$.
Moreover $\scrH$ is isomorphic to the bounded symmetric domain $\calB$ of type-III.
Then we have
\begin{align*}
&\{V\subset H_{\CC}\;|\; \dim{V}=n,\; V\gg 0,\; \langle V,V\rangle =0\}\cong \scrH\cong \calB,\\
& \{W\subset H_{\CC}\;|\; \dim{W}=n,\; W\ll 0,\; \langle W,W\rangle =0\}\cong \bar{\scrH}\cong \bar{\calB}.
\end{align*}
By \cite[\S 5.4]{FHW} or \cite[Proposition 2.5]{H} we have the isomorphism
$$\calM_D\stackrel{\sim}{\to}\calB\times \bar{\calB};\quad C_{(V,W)}\mapsto (V,W).$$
\begin{prop}\label{cyc-bd}
$\exp{(X)}C_0\in \calM_{\check{D}}$ is in the topological closure $\calM_{D}^{\cl}$ of $\calM_D$ in $\calM_{\check{D}}$ if $N^2=0$.
\end{prop}
\begin{proof}
Now we have the decomposition $H^{p,-p-1}=\bigoplus_{j=1,2,3} H^{p,-p-1}_j$ of (\ref{p-decomp}) with respect to $\exp{(iN)}\hat{F}$.
Since $XH^{p,-p-1}_1=XH^{p,-p-1}_3=0$, then
$$e^{X}H^{p,p-1}=H^{p,p-1}_1\oplus e^{X}H^{p,p-1}_2\oplus H^{p,p-1}_3.$$
Here the above components are orthogonal to each other, and by (\ref{val-1}) and (\ref{val-2}) $e^X H^{p,p-1}_2$ is non-negative (resp. non-positive) for $i^{2p-1}\langle \bullet ,\bar{\bullet} \rangle$ if $p$ is even (resp. odd).
Then $e^{X}H^{\ev}$ (resp. $e^XH^{\od}$) is in the closure $\calB^{\cl}$ (resp. $\bar{\calB}^{\cl}$) .
\end{proof}

\begin{rem}
If $N^2\neq 0$, $e^XC_0$ need not to be in $\calM_{D}^{\cl}$ (See \cite[proposition 4.6]{H}).
\end{rem}
\subsection{Maps to the Satake boundary components}
We call $(\sigma ,\exp{(\sigma_{\CC})F})$ a nilpotent orbit if it satisfies the following conditions:
\begin{itemize}
\item $\sigma$ is a cone in $\frakg$ where we can choose generators of $\sigma$ over $\QQ$;
\item Any elements of $\sigma$ are commutative with each other;
\item $F\in \check{D}$ and $\exp{(zN)}F\in D$ for $\Im{z}\gg 0$ and for $N$ in the relative interior $\sigma^{\circ}$;
\item $NF^{p}\subset F^{p-1}$ for $N\in\sigma$. 
\end{itemize}
The above conditions do not depend on the choice of $F\in\exp{(\sigma_{\CC})}F$.
By \cite{CK} the monodromy weight filtration $W(N)$ does not depend on the choice of $N\in\sigma^{\circ}$ (we denote it by $W(\sigma)$), and $(W(\sigma),F)$ is a LMHS  by \cite{S}.  
\begin{dfn}\label{ev-od}
A nilpotent orbit $(\sigma,\exp{(\sigma_{\CC})F})$ is called even-type (resp. odd-type) if $N^2=0$ for $N$ in the relative interior $\sigma^{\circ}$ and $I^{p,-p}=0$ for any odd (resp. even) integer $p$ with respect to the LMHS $(W(\sigma ),F)$.
A nilpotent cone $\sigma$ is called even-type (resp. odd-type) if every $\sigma$-nilpotent orbit is even-type (resp. odd-type).
A fan $\Sigma$ is called even-type (resp. odd-type) if any face of $\Sigma$ is even-type (resp. odd-type)
\end{dfn}
Let $(\sigma,\exp{(\sigma_{\CC})F})$ be a nilpotent orbit of even-type or odd-type.
For $(N,F)$ with $N\in \sigma^{\circ}$, we have the compact submanifold $C_0\in\calM_{D}$ as in the previous subsection.
By Proposition \ref{cyc-bd}, $\exp{(X)}C_0\in \calM_{D}^{\cl}$.
If $K_0\neq L_0$, $\calM_{D}\cong\calB\times \bar{\calB}$.
We then define the two projections
$$p^{\ev}:\calM_{D}^{\cl}\to \calB^{\cl},\quad p^{\od}:\calM_{D}^{\cl}\to \bar{\calB}^{\cl}.$$
If $K_0=L_0$, we define $p^{\ev}$ as the canonical isomorphism $\calM_{D}^{\cl}\stackrel{\sim}{\to}\calB^{\cl}$ and $p^{\od}$ as the complex conjugation map $\calM_{D}^{\cl}\stackrel{\sim}{\to}\calB^{\cl}\to\bar{\calB}^{\cl}$.   
\begin{thm}\label{cyc-sigma}
If the nilpotent orbit is of even-type (resp. odd-type),
$p^{\even}(\exp{(X)}C_0)\in \calB^{\cl}$ (resp. $p^{\odd}(\exp{(X)}C_0)\in \overline{\calB}^{\cl}$) does not depend on the choice of $N\in \sigma^{\circ}$ and $F\in\exp{(\sigma_{\CC})}F$.
\end{thm}
\begin{proof}
By Lemma \ref{exp-zN}, $u=\exp{(iN)}v\in H^{p,-p-1}_2$ for $v\in I^{p,-p}$.
By Lemma \ref{X-act}
\begin{align}\label{exp-X}
\exp{(X)}u=\exp{(iN)}v-\exp{(-iN)}v=2iNv\in e^XH^{p,-p-1}_2,
\end{align}
and for $u'=\exp{(iN)}\bar{v}\in H^{-p,p-1}_2$,
\begin{align*}
\exp{(X)}u'=\exp{(iN)}\bar{v}-\exp{(-iN)}\bar{v}=2iN\bar{v}\in e^{X}H^{-p,p-1}_2.
\end{align*}
Since $H^{p,-p-1}_3=0$ for even (resp. odd) $p$ by definition, then 
\begin{align*}
&e^X H^{\ev}=\bigoplus_{p: \text{ even}}(H^{p,p-1}_1\oplus e^XH^{p,p-1}_2)=\bigoplus_{p: \text{ even}}H^{p,p-1}_1\oplus \Im{N}_{\CC},\\
&(\text{resp. }e^X H^{\od}=\bigoplus_{p: \text{ odd}}H^{p,p-1}_1\oplus \Im{N}_{\CC}).
\end{align*}
Here $\Im{N}=W(\sigma)_{-2}$ does not depend on the choice of $N\in \sigma^{\circ}$.
Moreover $H^{p,-p-1}_1$ is in the kernel of the action of $\sigma_{\CC}$ and $L^{-1,-1}_{\RR}(W(\sigma),F)$.
Then $e^X H^{\even}$ (resp. $e^XH^{\od}$) does not depend on the choice of $F\in\exp{(\sigma_{\CC})}F$ and $N\in\sigma^{\circ}$.
\end{proof}
The Satake boundary components of $\calB$ is corresponding to the set of real isotropic subspaces of $H_{\RR}$ (\cite[Proposition 4.4]{N}).
We denote by $\BB_{\Sat}(\sigma)$ the Satake boundary component corresponding to the real vector space $W(\sigma)_{-2}$.
The boundary point $p^{\ev}(\exp{(X)}C_0)$ (resp. $p^{\od}(\exp{(X)}C_0)$) is contained in the Satake boundary component $\BB_{\Sat}(\sigma)$ (resp. $\overline{\BB_{\Sat}(\sigma)}$).
\begin{cor}\label{ev-map-sigma}
Let $\sigma$ be an even-type (resp. odd-type) nilpotent cone and let $\BB (\sigma)$ be the set of all $\sigma$-nilpotent orbits.
Then $p^{\ev}$ (resp. $p^{\od}$) gives a well-defined map $\BB(\sigma)\to \BB_{S}(\sigma)$ (resp. $\BB(\sigma)\to \overline{\BB_{S}(\sigma)}$).
\end{cor}
Let $\Sigma$ be an even-type (resp. odd-type) fan.
We write $D_{\Sigma}=\bigsqcup_{\sigma\in \Sigma}\BB(\sigma)$.
We have the well-defined map
\begin{align*}
p^{\ev}: D_{\Sigma}\to \scrH_S \quad (\text{resp. } p^{\od}: D_{\Sigma}\to \bar{\scrH}_S)
\end{align*}
where the map restricted on $\BB(\sigma)$ is given by Corollary \ref{ev-map-sigma}.
If $\sigma=\{0\}$, $\BB(\sigma)=D$ and $p^{\ev}|_D$ (resp. $p^{\od}|_{D}$) is the map given by taking the even-part (resp. odd-part) of the Hodge decomposition.
Let $\Gamma$ be a subgroup of $G_{\ZZ}=\Aut{(H_{\ZZ},\langle\; ,\; \rangle)}$.
$\gamma\in \Gamma$ gives a map from $\BB(\sigma)$ to $\BB(\Ad{(\gamma)}\sigma)$ by
$$(\sigma,\exp{(\sigma_{\CC})}F)\mapsto (\Ad{(\gamma)}\sigma ,\gamma\exp{(\sigma_{\CC})}F).$$ 
We assume that $\Gamma$ is compatible with the fan $\Sigma$.
Then $p^{\ev} $ (resp. $p^{\od}$) is compatible with the action of $\Gamma$ and we can define
\begin{align}\label{Sat-map}
p^{\ev}: \Gamma\bs D_{\Sigma}\to \Gamma\bs \scrH_S \quad (\text{resp. } p^{\od}: \Gamma\bs D_{\Sigma}\to \overline{\Gamma\bs\scrH _S}).
\end{align}
\begin{exm}\label{siegel-exm}
Let $D$ be a Siegel space $\frakh$, i.e. $h^{p,-p-1}=0$ if $p\neq 0,-1$.
We take a nilpotent cone $\sigma$ in the open real cone $\eta _i^+$ of \cite[\S 4]{CCK}.
By \cite[Proposition 4.2]{CCK}, $(W(\sigma),F)$ is a LMHS if and only if $F\in \exp{(\sigma_{\CC})}\frakh\subset\check{\scrH}$, and
$$\BB (\sigma)\cong \exp{(\sigma_{\CC})}\scrH /\exp{(\sigma_{\CC})}.$$
Here the LMHS is the following type:
$$
\xymatrix{
&\stackrel{(0,0)}{\bullet}\ar@{->}[dd]^N& \\ \stackrel{(0,-1)}{\bullet}& & \stackrel{(-1,0)}{\bullet} \\ &\stackrel{(-1,-1)}{\bullet}&
}$$
Then any nilpotent orbit for $\frakh$ is even-type, where $p^{\even}|_{D}=id$ and $p^{\even}:\BB (\sigma)\to \BB_S(\sigma)$ coincides with the map $\zeta$ of \cite[\S 4(3)]{CCK}.
By \cite[\S 6]{CCK} $\zeta $ induces the maps from the toroidal compactifications to the Satake compactification.
\end{exm}

\subsection{Maps to the toroidal boundary components.}
For an even-type (resp. odd-type) nilpotent orbit $(\sigma, \exp{(\sigma_{\CC})F})$, we have the $\RR$-split nilpotent orbit   $(\sigma, \exp{(\sigma_{\CC})}\hat{F})$.
Let $H_{\CC}=\bigoplus _{p+q=0,-1,-2} I^{p,q}$ be the Deligne decomposition with respect to the LMHS $(W(\sigma),\hat{F})$.
We define $\tilde{F}\in \check{\scrH}$ by
\begin{align*}
&\tilde{F}^{0}=(\bigoplus_{p: \text{ even}}I^{p,-p-1})\oplus (\bigoplus_{p}I^{p,-p}),\\
&(\text{resp. }\tilde{F}^{0}=(\bigoplus_{p: \text{ odd}}I^{p,-p-1})\oplus (\bigoplus_{p}I^{p,-p}))
\end{align*}
Now $(W(\sigma),\tilde{F})$ (resp. $(W(-\sigma),\overline{\tilde{F}})$) is a $\RR$-split LMHS.
In fact $\langle \bullet ,N\bullet\rangle$ (resp.  $\langle \bullet ,-N\bullet\rangle$) with $N \in \sigma^{\circ}$ gives a polarization on $\Gr^{W(\sigma)}_{0}$ since $i^{2p}=1$ if $p$ is even and $-1$ if $p$ is odd.
Then $(\sigma ,\exp{(\sigma_{\CC})}\tilde{F})$ is a $\RR$-split nilpotent orbit for $\scrH$ (resp. $\bar{\scrH}$).
Moreover we have the following proposition:
\begin{prop}\label{p-tilde}
$p^{\ev}(e^{zN}\hat{F})=e^{zN}\tilde{F}$ (resp. $p^{\od}(e^{\bar{z}N}\hat{F})=e^{\bar{z}N}\tilde{F}$) for $\Im{z}>0$ and $N\in\sigma^{\circ}$.
\end{prop}
\begin{proof}
Let $H_{\CC}=\bigoplus_p H^{p,-p-1}$ be the Hodge decomposition with respect to $\exp{(zN)}\hat{F}$.
Since $\sigma$ is even-type (resp. odd-type), $H^{p,-p-1}_3=0$ if $p$ is even (resp. odd).
Then by (\ref{decomp_123})
\begin{align*}
p^{\ev}(e^{zN}\hat{F})^0&=\bigoplus_{p:\text{ even}}(H^{p,-p-1}_1\oplus H^{p,-p-1}_2)\\
&=\bigoplus_{p: \text{ even}} I^{p,-p-1}\oplus e^{zN}(\bigoplus_p I^{p,-p})=e^{zN}\tilde{F}^0
\end{align*}
(resp. $p^{\od}(e^{\bar{z}N}\hat{F})^0=e^{\bar{z}N}\tilde{F}^0$).
\end{proof}
We denote by $\BB_{\tr}(\sigma)$ the boundary component for $\sigma$ of $\scrH_{\sigma}$.
For an even-type (resp. odd-type) nilpotent cone $\sigma$, we define the map $\tilde{p}^{\ev}:\BB(\sigma)\to \BB_{\tr}(\sigma)$ (resp. $\tilde{p}^{\od}:\BB(\sigma)\to \overline{\BB_{\tr}(\sigma)}$) by
$$(\sigma,\exp{(\sigma_{\CC})}F)\mapsto (\sigma,\exp{(\sigma_{\CC})}e^{i\delta}\tilde{F})$$
where $F=e^{i\delta}\hat{F}$.
Then for an even-type (resp. odd-type) fan $\Sigma$ we can define the map
$$\tilde{p}^{\ev}:D_{\Sigma}\to \scrH_{\Sigma}\quad (\text{resp. }\tilde{p}^{\od}:D_{\Sigma}\to \bar{\scrH}_{\Sigma}),$$
and for a subgroup $\Gamma$ of $G_{\ZZ}$ which is compatible with $\Sigma$ we have
$$\tilde{p}^{\ev}:\Gamma\bs D_{\Sigma}\to \Gamma\bs \scrH_{\Sigma}\quad (\text{resp. }\tilde{p}^{\od}:\Gamma\bs D_{\Sigma}\to \overline{\Gamma\bs \scrH}_{\Sigma}),$$
where $\tilde{p}^{\ev}=id$ if $D=\scrH$.
Now we have the map $\zeta :\Gamma\bs \scrH_{\Sigma}\to \Gamma\bs \scrH_{S}$ (resp. $\bar{\zeta} :\Gamma\bs \scrH_{\Sigma}\to \overline{\Gamma\bs \scrH}_{S}$) such that
\begin{align*}
\zeta\circ \tilde{p}^{\ev}(\sigma ,\exp{(\sigma_{\CC})F})&=W(\sigma)_{-2}\oplus (\bigoplus_{p:\text{ even}}I^{p,-p-1}) \\
&=p^{\ev}(\sigma ,\exp{(\sigma_{\CC})F})
\end{align*}
(resp. $\bar{\zeta}\circ \tilde{p}^{\od}(\sigma ,\exp{(\sigma_{\CC})F})=p^{\od}(\sigma ,\exp{(\sigma_{\CC})F)}$).
Then we have the following theorem:
\begin{thm}
$p^{\ev}$ (resp. $p^{\od}$) factors through $\tilde{p}^{\ev}$ (resp. $\tilde{p}^{\od}$).
\end{thm}

\section{The (1,1,1,1)-case}
In this section we consider the case where $h^{p,-p-1}=1$ if $p=1,0,-1,-2$, $h^{p,-p-1}=0$ otherwise.
In this case
$$G\cong Sp(2,\RR),\quad L_0\cong U(1)\times U(1),\quad K_0\cong U(2),\quad C_0\cong \PP^1.$$
Any nilpotent cone in this case is rank $1$, and its generator $N$ is classified as follows:
\begin{itemize}
\item[(I)] $N^2=0,\;\dim{(\Im{N})}=1$;
\item[(II)] $N^2=0,\;\dim{(\Im{N})}=2$;
\item[(III)] $N^3\neq 0,N^4=0$.
\end{itemize}
Here type-I is even-type, type-II is odd-type and type-III is neither.
The boundary components are described by \cite[\S 12.3]{KU} or \cite{GGK}.
For type-I (resp. type-II), we describe $p^{\ev} $ and $\tilde{p}^{\ev}$ (resp. $p^{\od}$ and $\tilde{p}^{\od}$) explicitly and show the continuity.

\subsection{Preliminary}
At first we describe the period domain $D$.
Let $H_{\ZZ}$ be a rank-$4$ $\ZZ$-module.
We define a bilinear form $\langle \; ,\; \rangle$ by
$$(\langle e_j,e_k\rangle)_{j,k}=\begin{pmatrix}
\LARGE{0}&\LARGE{-I}\\ \LARGE{I}&\LARGE{0}
\end{pmatrix}$$
for a basis $e_1,\ldots ,e_4$ of $H_{\ZZ}$.
We have an open immersion
\begin{align}\label{D_check_dsc}
&\mathrm{Sym}(2,\CC)\times \CC \hookrightarrow  \check{D};\quad (\tau,\lambda)\mapsto F(\tau ,\lambda)\quad \text{where}\\
&F^1(\tau,\lambda)=\Span_{\CC} \l\{ \begin{pmatrix}\tau_{12}\\ \tau_{22}\\ 0\\1\end{pmatrix}+\lambda \begin{pmatrix}
\tau_{11}\\ \tau_{21}\\1 \\ 0
\end{pmatrix}\r\},\quad F^0(\tau,\lambda)=\Span_{\CC} \l\{ \begin{pmatrix}\tau_{12}\\ \tau_{22}\\0 \\1\end{pmatrix}, \begin{pmatrix}
\tau_{11}\\ \tau_{21}\\1 \\ 0
\end{pmatrix}\r\},\nonumber\\
&F^{-1}(\tau,\lambda)=F^1(\tau,\lambda)^{\perp}.\nonumber
\end{align}
Here $F(\tau ,\lambda)\in D$ if, and only if,
\begin{align}\label{inD}
\det{(\Im{\tau})}<0,\quad -i \langle\omega,\bar{\omega}\rangle >0 \quad\text{for } 0\neq\omega\in F^1(\tau,\lambda).
\end{align}

Let $\sigma$ be a nilpotent cone and let $\Gamma$ be a subgroup of $G_{\ZZ}$ such that
\begin{align}\label{str-cpt}
\text{there exists } \gamma\in \Gamma\text{ which satisfies }\sigma=\RR_{\geq 0}\log{(\gamma)}.
\end{align}
We write $\Gamma(\sigma)^{\gp}=\exp{(\sigma_{\RR})}\cap \Gamma$.
The topology of $\GsDs$ is the quotient topology via the map $\torsor$ (\cite[3.4.2]{KU}).
Here $\Es$ is the subset of $\check{D}\times \CC$ such that
\begin{align*}
(F,z)\in E_{\sigma}\Leftrightarrow \begin{cases}(\sigma ,\exp{(\sigma_{\CC})}F)\text{ is a nilpotent orbit}&\text{if }z=0,\\ \exp{(\ell (z)N)}F\in D&\text{if }z\neq 0,\end{cases}
\end{align*}
where $N$ is the generator of $\exp{(\sigma)}\cap \Gamma$ and $\ell (z)=\log{z}/2\pi i$ (we may assume that $\ell(z)$ is the one-valued function by taking a branch of $\log$).
The map $\torsor$ is given by
$$(F,z)\mapsto \begin{cases}(\sigma ,\exp{(\sigma_{\CC})}F)&\text{if }z=0,\\ \exp{(\ell (z)N)}F\pmod{\Gs}&\text{if }z\neq 0.\end{cases}$$
The topology of $\Es$ is the {\it strong topology} in $\check{D}\times \CC$.
For a fan $\Sigma$ and a subgroup $\Gamma$ of $G_{\ZZ}$ which is compatible with $\Sigma$, the topology of $\Gamma\bs D_{\Sigma}$ is the strongest topology such that $\GsDs\to \Gamma\bs D_{\Sigma}$ is continuous for all $\sigma\in \Sigma$.

We review the definition of the strong topology briefly.
For an analytic space $Y$ and a subset $X$, a subset $U$ of $X$ is open in the strong topology in $Y$ if $f^{-1}(U)$ is open for any analytic space $Z$ and any analytic map $f:Z\to Y$ such that $f(Z)\subset X$.
The following example is typical:
\begin{exm}[{\cite[3.1.3]{KU}}]\label{str-top}
Let $Y=\CC^2$ and $X=Y-\{0\}\times \CC^*$.
The strong topology on $X$ in $Y$ does not coincide with the topology as a subspace of $Y$ around the origin.  
We put
\begin{align*}
&U_n(\delta_n)=\left\{\begin{array}{l|l}(z_1,z_2)\in\Delta_{\delta_n}^2&\begin{array}{ll}|z_1|^n<|z_2|, \; z_2\neq 0\end{array}\end{array}\right\},\\
&U(\delta)=\bigcup_n U_n(\delta_n)\cup \{(0,0)\}
\end{align*}
where $\Delta_{\delta_n}$ is the $\delta_n$-open disk with $\delta_n>0$ and $\delta=\{\delta_n\}_n$.
Then $U(\delta)\subset X$ is an open neighborhood of the origin and $U(\delta)$, where $\delta$ runs over all sequences in $\RR_{>0}$, form a fundamental system of neighborhoods of the origin. 
\end{exm}
\subsection{Boundaries for type-I}
Let $\sigma=\RR_{\geq 0} N$ be the type-I nilpotent element with $N(e_3)=e_1$ and $N(e_j)=0$ for $j=1,2,4$.
Here the LMHS of this type is described as the following diagram:
$$
\xymatrix{
&\stackrel{(0,0)}{\bullet}\ar@{->}[dd]^N& \\ \stackrel{(1,-2)}{\bullet}& & \stackrel{(-2,1)}{\bullet} \\ &\stackrel{(-1,-1)}{\bullet}&
}$$
Then this nilpotent cone is even-type.
We define a fan
$$\Sigma_{\ev}=\{\Ad{(g)}\sigma\; |\; g\in G_{\ZZ}\}.$$
Then $\Sigma_{\ev}$ is the fan of all nilpotent cones of even-type by \cite[\S 12.3]{KU}.

We write
\begin{align*}
\xi _0(w)=\begin{pmatrix}
0\\w\\1\\0
\end{pmatrix},\quad
\xi _1(v,w)=\begin{pmatrix}
w\\v\\0\\1
\end{pmatrix}
\end{align*}
and define a filtration $F(v,w)\in \check{D}$ by
\begin{align*}
F^{1}(v,w)=\Span_{\CC}\{\xi_1(v,w)\},\quad F^0(v,w)=\Span_{\CC}\{\xi_1 (v,w),\xi _0(w)\} .
\end{align*}
Then
\begin{align*}
\BB (\sigma )=\left\{\begin{array}{l|l}(\sigma, \exp{(\CC N)}F(v,w))&\end{array}\Im{v}<0,\;w\in \CC\right\}.
\end{align*}
Let $F=F(v,w)$ with $\Im{v}<0$.
Then the $(0,0)$-component of the Deligne decomposition for $(W(N),F)$ is
$$F^0\cap W_0(N)\cap (\bar{F}^0\cap W_0(N)+\bar{F}^{-1}\cap W_{-2}(N))=F^0\cap(\bar{F}^0+N\bar{F}^0),$$
which is generated by
$$e=\xi_0(w)-\gamma\xi_1(v,w)=\begin{pmatrix}-\gamma w\\ \Re{w}-\gamma \Re{v}\\ 1\\ -\gamma\end{pmatrix},\quad \text{where }\gamma=\frac{\Im{w}}{\Im{v}}.$$
The $\RR$-split mixed Hodge structure $(W(N),\hat{F})$ associated to  $(W(N),F)$ is given by
$$\hat{F}=\exp{(\gamma i\Im{w}N)}F.$$
In fact, the $(0,0)$-component is generated by
$$\hat{e}=\exp{(\gamma i\Im{w}N)}e=\begin{pmatrix}-\gamma\Re{w}\\ \Re{w}-\gamma\Re{v}\\1 \\-\gamma \end{pmatrix}.$$
Then for the Hodge decomposition with respect to $\exp{(iN)}\hat{F}$,
$$
u_1:=\xi_1(v,w)\in H^{1,-2},\quad u_0:=\exp{(iN)}\hat{e}\in H^{0,-1}.
$$
Here $H^{\even}$ is generated by $u_0$ and $\bar{u}_1$, and by (\ref{exp-X})
$$\exp{(X)}u_0=2iN\hat{e}=2i e_1, \quad\exp{(X)}u_1=u_1.$$
Then
\begin{align}\label{even-1}
p^{\even}(\sigma,\exp{(\sigma_{\CC})}F)&=\exp{(X)}H^{\even}=\Span_{\CC}\{e_1, \bar{v}e_2+e_4\},
\end{align}
which is contained in $\BB_S (\sigma)$.
Moreover the map $p^{\ev}:\BB(\sigma)\to \BB_S(\sigma)$ is surjective.

Now $\tilde{F}$ is given by $\tilde{F}^0=\Span_{\CC}\{\overline{\xi_1(v,w)}, \hat{e}\}$.
By Proposition \ref{p-tilde}, for $\Im{z}>0$
$$\tilde{p}^{\ev}(\exp{(zN)}\hat{F})=\exp{(zN)}\tilde{F}=\begin{pmatrix}z-\gamma i\Im{w}&\bar{w}\cr \bar{w}&\bar{v}\end{pmatrix}.$$
Since $\exp{(\sigma_{\CC})}F=\exp{(\sigma_{\CC})}\hat{F}$, we have
$$\tilde{p}^{\ev}(\sigma ,\exp{(\sigma_{\CC})}F)=(\sigma,\exp{(\sigma_{\CC})}\tilde{F}).$$

\begin{prop}\label{ev-cont}
Let $\Gamma$ be a subgroup of $G_{\ZZ}$ which is compatible with $\Gamma$ and satisfies the condition (\ref{str-cpt}) for any $\sigma\in\Sigma_{\ev}$.
Then $\tilde{p}^{\ev}:\Gamma\bs D_{\Sigma_{\ev}}\to \Gamma\bs \scrH_{\Sigma_{\ev}}$ is continuous.
\end{prop}
\begin{proof}
It is sufficient to show the continuity around the boundary point $(\sigma,\exp{(\sigma_{\CC})}F(v,w))$ in $\Gamma\bs D_{\Sigma_{\ev}}$.
We write 
$$F=F(v,w),\quad \xi_0=\xi_0(w),\quad \xi_1=\xi_1(v,w),$$
and we assume $F=\hat{F}$, i.e. $\Im{w}=0$ (it is similar to show the continuity for $F\neq \hat{F}$, and we omit the proof of it).
A neighborhood of $(\sigma ,\exp{(\sigma_{\CC})}\tilde{F})$ in $\Gamma\bs\scrH_{\Sigma_{\ev}}$ is written as
\begin{align}\label{op}
&\{\exp{\l(\ell (z'_4)N\r)}\tilde{F}(z')\; |\;z'\in\Delta_{\varepsilon}^4,\; z'_4\neq 0\}\\
&\sqcup \{\exp{(\sigma_{\CC})}\tilde{F}(z')\;|\; z'\in\Delta^4_{\varepsilon},\; z'_4= 0\}\nonumber
\end{align}
for sufficiently small $\varepsilon >0$ where
$$\tilde{F}(z')=\tilde{F}+\begin{pmatrix}z'_1&z'_2\\z'_2&z'_3
\end{pmatrix}.$$
We describe neighborhoods of $(\sigma,\exp{(\sigma_{\CC})}F)$ of $\Gamma\bs D_{\Sigma_{\ev}}$, and show that there is a small neighborhood whose image through $\tilde{p}^{\ev}$ is contained in the above neighborhood (\ref{op}).

A neighborhood of the boundary point $(\sigma ,\exp{(\sigma_{\CC})}F)$ in $\Gamma\bs D_{\Sigma_{\ev}}$ is given by $E_{\sigma}$ and the map $\phi: E_{\sigma}\to \Gamma \bs D_{\Sigma_{\ev}}$.
It is sufficient to show $\tilde{p}^{\ev}\circ\phi$ is continuous.
We describe a neighborhood of $(F,0)\in E_{\sigma}\subset \check{D}\times \CC$.
Let $\Delta$ be a small open disk.
By (\ref{D_check_dsc}) an open neighborhood of $F$ in $\check{D}$ is given by
\begin{align*}
&\Delta^4 \hookrightarrow \check{D};\quad (z_1,\ldots ,z_4)\mapsto F(z)\\
&\text{where }F^1(z)=\Span_{\CC}\l\{\xi_1+ \theta_1(z)+z_4\l( \xi_0+\theta_0 (z)\r) \r\},\\
&F^0(z)=\Span_{\CC}\l\{\xi_1 + \theta_1(z), \xi_0 +\theta_0(z)\r\}
\end{align*}
where
\begin{align*}
\theta_0(z)= \begin{pmatrix}
z_1\\z_2\\0\\0
\end{pmatrix},\quad \theta_1(z)=
\begin{pmatrix}
z_2\\z_3\\0\\0
\end{pmatrix}.
\end{align*}
Then we have an open neighborhood $\Delta^5\hookrightarrow \check{D}\times \CC$ by $(z_1,\ldots ,z_5)\mapsto (F(z_1,\ldots ,z_4),z_5)$.
Here
\begin{align*}
(F(z_1,\ldots ,z_4),z_5)\in E_{\sigma}\Leftrightarrow \begin{cases}z_4=0&\text{if }z_5=0,\\ \exp{(\ell (z_5)N)}F(z_1,\ldots ,z_4)\in D&\text{if }z_5\neq 0.\end{cases}
\end{align*}
For $z_5\neq 0$, we write
\begin{align*}
&\eta_0(z)=e^{\ell (z_5)N}(\xi_0+\theta_0 (z))=e^{\ell (z_5)N}\xi_0+\theta_0 (z),\\
&\eta_1(z)=e^{\ell (z_5)N}(\xi_1+\theta_1 (z))+z_4\eta_0(z)=\xi_1 +\theta_1 (z)+z_4\eta_0(z).
\end{align*}
Then
\begin{align*}
&e^{\ell(z_5)N}F^1(z)=\Span_{\CC}\{\eta_1(z)\},\quad e^{\ell(z_5)N}F^0(z)=\Span_{\CC}\{\eta_1(z),\eta_0 (z)\}.
\end{align*}
If $z_1,\ldots ,z_5\to 0$ provided that $|z_4|\log{|z_5|}\to 0$, we then have the convergences
\begin{align}\label{conv-ev}
&\eta_1(z)\to \xi_1, \quad\eta_0(z)-e^{\ell (z_5)N}\xi_0\to 0\\
&\langle\overline{\eta_1(z)}, \eta_1(z)\rangle \to \langle \overline{\xi_1} ,\xi_1 \rangle,\quad \langle \overline{\eta_1(z)},\eta_0(z)\rangle -\langle \overline{\xi_1},e^{\ell(z_5)N}\xi_0\rangle \to 0.\nonumber
\end{align}
Here
$$e^{\ell (z_5)N}F^1=\Span_{\CC}\{\xi_1\},\quad e^{\ell (z_5)N}F^0=\Span_{\CC}\{\xi_1,e^{\ell(z_5)N}\xi_0\}.$$
Then 
$$e^{\ell(z_5)N}F(z)-e^{\ell (z_5)N}F\to 0$$
if $z_1,\ldots ,z_5\to 0$ provided that $|z_4|\log{|z_5|}\to 0$.
Therefore, by the conditions (\ref{inD}) of $D$,
$$U_n(\delta_n)=\left\{\begin{array}{l|l}(z_1,\ldots ,z_5)\in\Delta^5&\begin{array}{ll}|z_4|^n<|z_5|, \\ |z_4|,|z_5|<\delta_n, \; z_5\neq 0\end{array}\end{array}\right\}\subset E_{\sigma}$$
if $\delta_n$ is sufficiently small for $n\geq 1$.
Here $U_n(\delta_n)$ is a small neighborhood of $\{e^{\ell(z_5)N}F\; |\;|z_5|< \delta_n\}$.
By Example \ref{str-top} and the definition of strong topology,
$$U(\delta)=\l(\bigcup_n U_n(\delta_n)\r)\sqcup \{z\in \Delta^5\; |\; z_4=z_5=0\}$$
is an open neighborhood of $(0,F)$ in the strong topology of $E_{\sigma}$ taking $\delta=\{\delta_n\}$ and $\Delta$ sufficiently small.

Next we show $p^{\ev}(e^{\ell(z_5)N}F(z))$ approaches $e^{\ell(z_5)N}\tilde{F}$ if $z_1,\ldots ,z_5\to 0$ provided that $z\in U_n(\delta_n)$, i.e.  $|z_4|\log{|z_5|}\to 0$.
For the Hodge decomposition for $e^{\ell (z_5)N}F(z)$ with $(z_1,\ldots ,z_5)\in U_n(\delta_n)$, the $(1,-2)$-component is generated by $\eta_1(z)$ and the $(0,-1)$-component is generated by
$$\alpha(z)=\langle \overline{\eta_1(z)},\eta_0(z)\rangle\; \eta_1(z)-\langle \overline{\eta_1(z)},\eta_1(z)\rangle\;\eta_0(z)$$
since $\overline{\eta_1(z)}\perp\alpha(z)$.
Then 
$$p^{\ev}(e^{\ell (z_5)N}F(z))=\Span_{\CC}\{\overline{\eta_1(z)}, \alpha (z)\}.$$
We write
$$\beta(z_5)=\langle \overline{\xi_1},e^{\ell(z_5)N}\xi_0\rangle\; \xi_1-\langle \overline{\xi_1},\xi_1\rangle\;e^{\ell (z_5)}\xi_0,$$
which is in the $(0,-1)$-component of the Hodge decomposition of $e^{\ell(z_5)N}F$.
Then, by Proposition \ref{p-tilde}, 
$$\Span_{\CC}\{\overline{\xi_1}, \beta(z_5) \}=p^{\ev}(e^{\ell(z_5)N}F)=e^{\ell(z_5)N}\tilde{F}.$$
By the convergence (\ref{conv-ev}) we have 
$$\eta_1(z)\to \xi_1,\quad \alpha(z)-\beta(z_5)\to 0$$ 
if $z_1,\ldots ,z_5\to 0$ provided that $z\in U_n(\delta_n)$, moreover we have
$$p^{\ev}(e^{\ell (z_5)N}F(z))-e^{\ell(z_5)N}\tilde{F}\to 0$$
in the upper half space. 
Then $U(\delta)$ is contained in the neighborhood (\ref{op}) if $\delta_n$ is sufficiently small.
\end{proof}
\begin{cor}
$p^{\ev}$ is continuous.
\end{cor}
\begin{rem}\label{rem-ev}
The $G_{\ZZ}(\BB_{S}(\sigma))$-admissible polyhedral decomposition (\cite[Definition 6.1]{CCK}) is the fan $\{\sigma,\{0\}\}$.
Then for a $\Gamma$-admissible decomposition $\Sigma_{\tr}$ (\cite[Definition 6.2]{CCK}), we have an injection $\Sigma_{\ev}\hookrightarrow \Sigma_{\tr}$.
Therefore we have the following commutative diagram:
\begin{align*}
\xymatrix{
\Gamma\bs D_{\Sigma_{\ev}}\ar@{->}[dr]_{p^{\ev }}\ar@{->}[r]^{\tilde{p}^{\ev}} & \Gamma\bs \scrH_{\Sigma_{\ev}}\ar@{->}[d]\ar@{->}[r]^{\subset}& \Gamma\bs \scrH_{\Sigma_{\tr}}\ar@{->}[ld]^{\zeta}\\
&\Gamma\bs\scrH_{\Sat}
}
\end{align*}
\end{rem}

\subsection{Boundaries for type-II}
Let $\sigma_m=\RR_{\geq 0}N_m$ be the type-II nilpotent cone with $N_me_3=-e_1$ and $N_me_4\mapsto -me_2$ where $m$ is a square-free positive integer.
The LMHS of this type is described as the following diagram:
$$\xymatrix{
\stackrel{(1,-1)}{\bullet}\ar@{->}[dd]^N&& \stackrel{(-1,1)}{\bullet}\ar@{->}[dd]^N\\ \\\stackrel{(0,-2)}{\bullet}&& \stackrel{(-2,0)}{\bullet}
}$$
Then this nilpotent cone is odd-type.
We define
$$\Sigma_{\od}=\{\Ad{(g)}\sigma_m\; |\; g\in G_{\ZZ}, \;m: \text{square-free positive integer}\}.$$
Then $\Sigma_{\od}$ is the fan of all nilpotent cones of odd-type by \cite[\S 12.3]{KU}.

Let $N=N_m$ and $\sigma=\sigma_m$.
We write
\begin{align*}
\xi _0^{\pm}=\begin{pmatrix}-1\\ \pm i\sqrt{m}\\ 0\\0\end{pmatrix},\quad \xi _1^{\pm}(w)=\begin{pmatrix}0\\w\\ \pm i\sqrt{m} \\1\end{pmatrix}.
\end{align*}
Then for $z\in \CC$
\begin{align*}
N\xi_1^{\pm}(w)=\pm i\rt-m\xi_0^{\pm},\quad \xi_1^{\pm}(w)\pm zi\rt-m \xi _0^{\pm}=\exp{(zN)}\xi _1^{\pm}(w). 
\end{align*}
We define a filtration $F_{\pm}(w)\in\check{D}$ by
\begin{align*}
F^1_{\pm}(w)=\Span_{\CC}\{\xi_{1}^{\pm}(w)\},\quad F^0_{\pm}(w)=\Span_{\CC} \{ \xi_1^{\pm}(w),\xi_0^{\pm}\}.
\end{align*}
Then
$$\BB(\sigma)=\{(\sigma ,\exp{(\CC N)}F_{+}(w))\; |\;w\in \CC\}\sqcup\{(\sigma ,\exp{(\CC N)}F_{-}(w))\; |\;w\in \CC\}.$$

Let $F=F_{\pm}(w)$ and $\sigma=\RR_{\geq 0}N$.
Then the $(1,-1)$-component of the Deligne decomposition of $(W(N),F)$ is generated by $\xi_1^{\pm}(w)$, and the $(-1,1)$-component is
$$F^{-1}\cap(\bar{F}^1+\bar{F}^0\cap W_{-2}(N))=F^{-1}\cap(\bar{F}^1+N\bar{F}^1),$$
which is generated by
\begin{align*}
\omega&=\langle \xi_1^{\pm}(w),\overline{\xi_0^{\pm}}\rangle\; \overline{\xi_1^{\pm}(w)}- \langle \xi_1^{\pm}(w),\overline{\xi_1^{\pm}(w)}\rangle\; \overline{\xi_0^{\pm}}\\
&=\mp 2i\rt-m\;\overline{\xi^{\pm}_1(w)}+2i\Im{w}\;\overline{\xi^{\pm}_0}\\
&=\mp2i\rt-m\exp{\l(-\frac{\Im{w}}{m}iN\r)}\;\overline{\xi_1^{\pm}(w)}.
\end{align*}
The $\RR$-split MHS $(W(N),\hat{F})$ associated to $(W(N),F)$ is given by
$$\hat{F}=\exp{\l(\frac{\Im{w}}{2m}iN\r)}F.$$
In fact, the $(1,-1)$-component is generated by
\begin{align*}
\hat{\xi}=\exp{\l(\frac{\Im{w}}{2m}iN\r)}\xi^{\pm}_1(w)
\end{align*}
and the $(-1,1)$-component is generated by
\begin{align*}
&\hat{\omega}=\exp{\l(\frac{\Im{w}}{2m}iN\r)}\omega= \mp2i\rt-m\exp{\l(-\frac{\Im{w}}{2m}iN\r)}\;\overline{\xi_1^{\pm}(w)}. 
\end{align*}
Then the Hodge decomposition for $\exp{(iN)}\hat{F}\in D$ is given by
\begin{align*}
u_1=\exp{(iN)}\hat{\xi}\in H^{1,-2},\quad \bar{u}_0=\exp{(iN)}\hat{\omega}\in H^{-1,0}.
\end{align*}
Here $H^{\od}$ is generated by $u_1$ and $\bar{u}_0$, and by (\ref{exp-X})
\begin{align*}
\exp{(X)}u_1=2iN\hat{\xi},\quad \exp{(X)}\bar{u}_0=2iN\hat{\omega}.
\end{align*}
Then
\begin{align*}
p^{\od}(\sigma ,\exp{(\sigma_{\CC})}F)=\exp{(X)}H^{\od}=\Span_{\CC}\{e_1,e_2\} ,
\end{align*}
which is contained in $\overline{\BB _S(\sigma)}$.
Moreover the map $p^{\od}:\BB (\sigma)\to \overline{\BB _S (\sigma)}$ is surjective.

Now $\tilde{F}\in \check{\scrH}$ is given by $\tilde{F}^0=\Span_{\CC}\{\hat{\xi},\hat{\omega}\}$.
By Proposition \ref{p-tilde}, for $\Im{z}>0$
\begin{align*}
p^{\odd}(\exp{(zN)}F)&=\exp{(zN)}\tilde{F}=\begin{pmatrix}-z&\pm\frac{\Im{w}}{2\sqrt{m}}\cr \pm\frac{\Im{w}}{2\sqrt{m}}&\Re{w}-mz\end{pmatrix}.
\end{align*}
Since $\exp{(\sigma_{\CC})}F=\exp{(\sigma_{\CC})}\hat{F}$, we have
$$\tilde{p}^{\od}(\sigma ,\exp{(\sigma_{\CC})}F)=(\sigma ,\exp{(\sigma_{\CC})}\tilde{F}).$$
Here $\tilde{p}^{\od}$ is {\it not} surjective. 

\begin{prop}
Let $\Gamma$ be a subgroup of $G_{\ZZ}$ which is compatible with $\Gamma$ and satisfies the condition (\ref{str-cpt}) for any $\sigma\in\Sigma_{\od}$.
Then $\tilde{p}^{\od}:\Gamma\bs D_{\Sigma_{\od}}\to\overline{\Gamma\bs \scrH}_{\Sigma_{\od}}$ is continuous.
\end{prop}
\begin{proof}
As in the proof of Proposition \ref{ev-cont}, it is sufficient to show the continuity around the boundary point $(\sigma,\exp{(\sigma_{\CC})}F_{\pm}(w))$ in $\Gamma\bs D_{\Sigma_{\od}}$.
We write 
$$F=\hat{F}_{\pm}(w),\quad \xi_0=\xi_0^{\pm},\quad \xi_1=\xi^{\pm}_1(w),$$
and we assume $F=\hat{F}$, i.e. $\Im{w}=0$ (it is similar to show the continuity for $F\neq \hat{F}$ and we omit the proof of it).
We describe neighborhoods of $(\sigma,\exp{(\sigma_{\CC})}F)$ of $\Gamma\bs D_{\Sigma_{\od}}$, and show the continuity.

A neighborhood of a boundary point $(\sigma ,\exp{(\sigma_{\CC})}F)$ in $\Gamma\bs D_{\sigma}$ is given by $E_{\sigma}$ and the map $\phi: E_{\sigma}\to \Gamma \bs D_{\Sigma_{\od}}$.
It is sufficient to show $\tilde{p}^{\od}\circ\phi$ is continuous.
We describe a neighborhood of $(F,0)\in E_{\sigma}\subset \check{D}\times \CC$.
An open neighborhood of $F$ in $\check{D}$ is given by
\begin{align*}
&\Delta^4 \hookrightarrow \check{D};\quad (z_1,\ldots ,z_4)\mapsto F(z)\\
&\text{where }F^1(z)=\Span_{\CC}\l\{\xi_1+ \theta_1(z)+z_4\l( \xi_0 +\theta_0 (z)\r) \r\},\\
&F^0(z)=\Span_{\CC}\l\{\xi_1+ \theta_1(z), \xi_0 +\theta_0(z)\r\}
\end{align*}
where
\begin{align*}
\theta_0(z)= \begin{pmatrix}
0\\z_2\\z_1\\0
\end{pmatrix},\quad \theta_1(z)=
\begin{pmatrix}
0\\z_3\\z_2\\0
\end{pmatrix}.
\end{align*}
Then we have an open neighborhood $\Delta^5\hookrightarrow \check{D}\times \CC$ by $(z_1,\ldots ,z_5)\mapsto (F(z_1,\ldots ,z_4),z_5)$.
Here
\begin{align*}
(F(z_1,\ldots ,z_4),z_5)\in E_{\sigma}\Leftrightarrow \begin{cases}z_1=z_2=0&\text{if }z_5=0,\\ \exp{(\ell (z_5)N)}F(z_1,\ldots ,z_4)\in D&\text{if }z_5\neq 0.\end{cases}
\end{align*}
For $z_5\neq 0$, we write
\begin{align*}
&\eta_0(z)=e^{\ell (z_5)N}(\xi_0+\theta_0 (z))=\xi_0+e^{\ell (z_5)N}\theta_0 (z),\\
&\eta_1(z)=e^{\ell (z_5)N}(\xi_1+\theta_1 (z))+z_4\eta_0 (z).
\end{align*}
Then
\begin{align*}
&e^{\ell(z_5)N}F^1(z)=\Span_{\CC}\{\eta_1(z)\},\quad e^{\ell(z_5)N}F^0(z)=\Span_{\CC}\{\eta_1(z),\eta_0 (z)\}.
\end{align*}
If $z_1,\ldots ,z_5\to 0$ provided that $|z_1|\log{|z_5|}\to 0$ and $|z_2|\log{|z_5|}\to 0$, then we have
\begin{align}\label{conv-od}
&\eta_1(z)-e^{\ell (z_5)N}\xi_1\to 0,\quad \eta_0(z)\to\xi_0,\\
&\langle\overline{\eta_1(z)}, \eta_1(z)\rangle- \langle \overline{e^{\ell (z_5)N}\xi_1},e^{\ell (z_5)N}\xi_1\rangle\to 0,\nonumber\\
&\langle\overline{\eta_1(z)}, \eta_0(z)\rangle\to \langle \overline{e^{\ell (z_5)N}\xi_1},\xi_0\rangle=\langle \overline{\xi_1}, \xi_0\rangle, \nonumber
\end{align}
Here
$$e^{\ell (z_5)N}F^1=\Span_{\CC}\{e^{\ell (z_5)N}\xi_1\},\quad e^{\ell (z_5)N}F^0=\Span_{\CC}\{e^{\ell (z_5)N}\xi_1,\xi_0\}.$$
Then 
$$e^{\ell(z_5)N}F(z)-e^{\ell (z_5)N}F\to 0$$
if $z_1,\ldots ,z_5\to 0$ provided that $|z_1|\log{|z_5|}\to 0$ and $|z_2|\log{|z_5|}\to 0$.
Therefore, by the conditions (\ref{inD}) of $D$,
$$U_{n,m}(\delta_{n,m})=\left\{\begin{array}{l|l}(z_1,\ldots ,z_5)\in\Delta^5&\begin{array}{ll}|z_1|^n<|z_5|, \;|z_2|^m<|z_5|,\\|z_1|,|z_2|,|z_5|<\delta_{n,m},\;   z_5\neq 0\end{array}\end{array}\right\}\subset E_{\sigma}$$
if $\delta_{n,m} $ and $\Delta$ is sufficiently small.
By Example \ref{str-top} and the definition of strong topology,
$$U(\delta)=\l(\bigcup_{n,m}U_{n,m}(\delta_{n,m})\r)\sqcup \{z\in \Delta^5\; |\; z_1=z_2=z_5=0\}$$
is an open neighborhood of $(F,0)$ in $E_{\sigma}$ for $\delta=\{\delta_{n,m}\}$.

As in the proof of the case for even-types, we can show that
$$p^{\od}(e^{\ell(z_5)N}F(z))-e^{\ell(z_5)N}\tilde{F}\to 0$$
if $z_1,\ldots ,z_5\to 0$ provided that $z\in U_{n,m}(\delta_{n,m})$.
Then $\tilde{p}^{\od}\circ \phi(U(\delta))$ is contained in the open neighborhood (\ref{op}) if $\delta_{n,m}$ is sufficiently small.
\end{proof}
\begin{cor}
$p^{\od}$ is continuous.
\end{cor}
\begin{rem}
A $G_{\ZZ}(\BB_{S}(\sigma))$-admissible decomposition $\Sigma$ does not contain $\sigma$ in this case.
Then we do not have a diagram like the one in Remark \ref{rem-ev}.
\end{rem}

\end{document}